\documentclass[11pt]{article}
\usepackage{amsmath,amsthm,amsfonts,amssymb,bm,mathrsfs,wasysym}
\usepackage{epsfig}
\usepackage[usenames]{color}
\usepackage{verbatim}
\usepackage{hyperref}
\usepackage{multicol}
\usepackage{comment}
\usepackage{float}
\usepackage{graphicx}
\usepackage[utf8]{inputenc}

\def\bs{\backslash}
\def\cU{\mathcal{U}}

\def\bs{\backslash}

\parskip \medskipamount
\newtheorem{theorem}{Theorem}[section]

\numberwithin{equation}{section}
\newtheorem{lemma}[theorem]{Lemma}
\newtheorem{proposition}[theorem]{Proposition}
\newtheorem{remark}[theorem]{Remark}

\numberwithin{equation}{section}

\def\N{\mathbb{N}}
\def\Z{\mathbb{Z}}

\def\R{\mathbb{R}}
\def\C{\mathcal{C}}
\def\V{\mathcal{V}}

\def\cL{\mathcal{L}}

\def\cR{\mathcal{R}}
\def\cO{\mathcal{O}}

\def\cF{\mathcal{F}}
\def\cE{\mathcal{E}}

\def\NN{\mathcal{N}}

\def\bE{\mathbb{E}}
\def\XX{\mathcal{X}}
\def\bP{\mathbb{P}}

\renewcommand{\phi}{\varphi}
\renewcommand{\epsilon}{\varepsilon}

\def\G{\mathcal{G}}

\def\cII{\mathcal{I}}
\allowdisplaybreaks

\newcommand{\1}{{\text{\Large $\mathfrak 1$}}}

\newcommand{\var}{\operatorname{var}}

\def\reff#1{(\ref{#1})}

\newcommand{\tn}{|\kern-.1em|\kern-0.1em|}

\def\bP{\mathbb{P}}
\def\capa{\text{cap}}

\def\tilde{\widetilde}

\begin{document}
\title{\bf Extracting subsets maximizing capacity \\
and Folding of Random Walks}

\author{Amine Asselah \thanks{
LAMA, Univ Paris Est Creteil, Univ Gustave Eiffel, UPEM, CNRS, F-94010, Cr\'eteil, France; amine.asselah@u-pec.fr} \and
Bruno Schapira\thanks{Aix-Marseille Universit\'e, CNRS, Centrale Marseille, I2M, UMR 7373, 13453 Marseille, France;  bruno.schapira@univ-amu.fr} 
}
\date{}
\maketitle
\begin{abstract}
We prove that in any finite set of $\Z^d$ with $d\ge 3$, 
there is a subset whose capacity and volume are both
of the same order as the capacity of the initial set. 
As an application we obtain estimates on the probability
of {\it covering uniformly} a finite set, and characterize some {\it folding} events, under optimal hypotheses. 
For instance, knowing that a region of space has an {\it atypically high occupation density} by some random walk, 
we show that this random region is most likely ball-like.\\

MSC 2010 \emph{subject classifications.} Primary 60F05, 60G50.
\end{abstract}

\begin{center}
Nous montrons que de tout ensemble fini de $\Z^d$ en dimension
trois et plus, on peut extraire un sous-ensemble dont la capacité
et le volume sont d'ordre de la capacité de l'ensemble initial.
Cette observation nous permet d'obtenir, sous des hypothèses
optimales, des estimées de la probabilité qu'une marche aléatoire
{\it recouvre uniformément} un ensemble fini, et
de caractériser certains événements de repliement de la marche.
Par exemple, lorsqu'on sait que la marche produit une densité
d'occupation grande dans une région, alors celle-ci a la forme d'une boule.
\newline
\newline
\emph{Mots cl\'es.} Temps Locaux; capacit\'e; marche al\'eatoire.
\end{center}

\section{Introduction}\label{sec-intro}
This note deals with {\it capacity} in the context of a random
walk on $\Z^d$, with $d\ge 3$. 
If $\bP_x$ is the law of the random walk starting from $x$, 
and $H_\Lambda^+$ is its return time to $\Lambda\subset \Z^d$,
then the capacity of $\Lambda$ is
\begin{equation}\label{def-capa}
\capa(\Lambda):=\sum_{x\in \Lambda} \bP_x\big(H^+_{\Lambda}=\infty\big).
\end{equation}

Our main observation is that in any finite set of $\Z^d$, say made of
disjoint balls with common radius $r$, there exists a subset whose size
and capacity are both of order the capacity of the initial set.
To state precisely this result, let us introduce the needed
notation. For $x\in \Z^d$, and $r\ge 1$, 
we define $B_r(x)= \{z\in \Z^d : \|z-x\|< r\}$, 
with $\|\cdot \|$ the Euclidean norm, and for $\C\subset \Z^d$, we let 
$B_r(\C):= \cup_{x\in \C} B_r(x)$. 

\begin{theorem}\label{prop-main} Assume $d\ge 3$.
There exists $\alpha>0$, 
such that for any $r\ge 1$ and any finite $\C\subset \Z^d$, there is
a subset $U\subseteq \C$, satisfying 
\begin{equation}\label{ineq-cor}
(i)\quad\capa(B_r(U))\ge \alpha\cdot  r^{d-2} |U|\quad\text{and}\quad
(ii)\quad r^{d-2} |U|\ge\alpha\cdot \capa(B_r(\C)).
\end{equation}
\end{theorem}
We now present two applications of this result, Theorems \ref{theo-AS15} and \ref{prop-folding} below. 
The former deals informally with the event of {\it covering uniformly} a fraction $\rho$ of a set, and bounds  
the probability of such event   
by exponential minus $\rho$ times the capacity of the set, 
under some optimal assumptions on $\rho$ and the scale 
at which we measure density occupation, 
thus improving upon Proposition 1.7 from~\cite{AS15}.
The latter, Theorem~\ref{prop-folding}, deals with the shape of the 
folding region for a walk conditioned on squeezing part of its range, and shows that this region is typically ball-like in
the sense that its capacity is of smallest possible order, 
that is with capacity of order its volume to the power $1-2/d$, 
as it is for balls.
This has some natural applications in the context of moderate deviations 
for the volume or the capacity of 
the range of the walk, as shown in \cite{AS15,AS19,AS20}.  

Let us mention that Theorem~\ref{prop-main} has found application 
in the context of Random Interlacements \cite{S20}.

To be more precise now, for $\Lambda\subset \Z^d$ made of
disjoint balls of radius $r$, consider the event
obtained by asking the random walk to spend
a time $\rho\cdot r^d$ in each ball making $\Lambda$, for some $\rho>0$. 
We have shown in~\cite{AS15} how to relate
the probability of such covering event with the capacity of $\Lambda$.
We let $\{S_n\}_{n\in \N}$ denote the simple random walk, $\bP$ its law when starting from the origin, and for
$n\in \N\cup\{\infty\}$, and $z\in \Z^d$, its local time is 
\begin{equation}\label{def-local}
\ell_n(z) := \sum_{k=0}^n\1\{S_k=z\} \quad \text{and for}
\ \Lambda\subset \Z^d, \ \ell_n(\Lambda):= \sum_{z\in \Lambda} \ell_n(z).
\end{equation}

\begin{theorem}\label{theo-AS15}
Assume $d\ge 3$.
There exist positive constants $A$ and $\kappa$, such that for any $r\ge 1$ and $\rho>0$ satisfying 
\begin{equation}\label{best-as15-hyp}
\rho r^{d-2}>A,
\end{equation}
one has for any finite $\C\subset \Z^d$
\begin{equation}\label{best-as15}
\bP\big(\ell_\infty(B_r(x))> \rho r^d\quad \forall x\in \C\big)\le
 \exp\big(-\kappa \cdot \rho\cdot \capa(B_r(\C) )\big).
\end{equation}
\end{theorem} 
The condition \reff{best-as15-hyp} improves upon
our previous condition $\rho r^{d-2}>|\C|^{2/d}$, from \cite{AS15}, and is optimal since 
typically a walk spends a time of order $r^2$ 
in a ball of radius $r$, conditionally on visiting it. 
In the case $r=1$ (when $B_r(x)=\{x\}$ for all $x$), we obtain in fact a stronger and more general result. Indeed, first the result holds true for any $\rho>0$ and we 
show that essentially in this case we can take the  constant $\kappa$ equal to one in~\eqref{best-as15}. Furthermore, we can also deal
with non-uniform covering and general transient walks. We refer to Theorem~\ref{theo.1} in Section~\ref{sec-AS15}  for a precise statement. 

\begin{remark}\emph{We note that 
Sznitman obtained results with a similar flavor as Theorem \ref{theo-AS15}
in the context of the Gaussian Free Field (GFF) 
and in the model of Random Interlacements,  respectively in  
\cite[Corollaty 4.4]{S15} and  \cite[Theorem 4.2]{S17} (see also \cite{LS15} for related results). 
}
\end{remark}

Our second application deals with finite times. 
For $r\ge 1$, $\rho>0$, $n\ge 1$, and $\C\subset \Z^d$ finite, we consider
\begin{equation}\label{def-folding}
\cF_n(r,\rho,\C):=\{\forall x\in \C,\quad \ell_n(B_r(x))>\rho r^d\}.
\end{equation}
In many folding problems, one central issue is
to characterize the size and the shape of the folding region $\C$, which might be random.
More precisely, one may consider folding events of the form
$\cup_{\C}\cF_n(r,\rho,\C)$, where the union is over all $\C\subseteq [-n,n]^d$, 
with only a lower bound on their volume, say $|\C|\ge L$. Then Theorem \ref{theo-AS15} and a naive union bound gives  
$$\bP(\cup_{\C}\cF_n(r,\rho,\C))\le (2n)^{d\cdot L}\cdot
\exp(-\kappa \rho\cdot a\cdot r^{d-2} L^{1-2/d}),$$ 
using also \reff{bounds-capa}, which is useful only when 
\begin{equation}\label{old-cond}
\rho \cdot r^{d-2} \ge CL^{2/d}\cdot \log(n).
\end{equation}
Now Theorem~\ref{prop-main} allows to go beyond this condition \reff{old-cond}, and gives 
$$\bP(\cup_{\C}\cF(r,\rho,\C))\le 
\exp(-\kappa \rho\cdot a \cdot r^{d-2} L^{1-2/d}),$$
under the weaker assumption: 
$$\rho \cdot r^{d-2} \ge C \log(n).$$  
The latter is of crucial importance in \cite{AS20}, and can also be used to characterize 
the shape of a localization region for a random walk, 
which we now describe in details. First, we introduce more notation.
To obtain a neat partition of $\Z^d$ we switch to cubes, rather than
balls. Define for $r\ge 1$, and $x\in\Z^d$, 
$$
Q_r(x):=[x-r/2,x+r/2)^d\cap \Z^d. 
$$
Define further for $\rho>0$ and $n\ge 1$, 
\begin{equation}\label{def-C}
\C_n(r,\rho):= \{x\in r\Z^d\, :\, \ell_n(Q_r(x))\ge \rho r^d\},
\quad \text{and}\quad
\V_n(r,\rho):=\bigcup_{x\in \C_n(r,\rho)} Q_r(x).
\end{equation}
We can now state our third result.

\begin{theorem}\label{prop-folding}
Assume $d\ge 3$. 
There are positive constants $\underline \kappa$, $\overline \kappa$, and $C$, such that for any $n$, $r$ and $L$ positive integers and $\rho>0$,  satisfying  
\begin{equation}\label{cond-folding}
\rho r^{d-2}\ge C\cdot \log(n), \quad \text{and}\quad n\ge C\rho r^dL,
\end{equation}
one has 
\begin{equation}\label{ineq-folding}
\exp\big(-\underline \kappa\cdot 
\rho \cdot r^{d-2}\cdot L^{1-2/d}\big) \le \bP\big( |\C_n(r,\rho)|> L\big)\le \exp\big(-\overline \kappa\cdot 
\rho \cdot r^{d-2}\cdot L^{1-2/d}\big).
\end{equation}
In addition there exists $A>0$, such that 
\begin{equation}\label{shape-folding}
\lim_{n\to\infty} \inf_{(r,\rho,L)} 
\bP\big(\capa(\V_n(r,\rho))\le A\cdot |\V_n(r,\rho)|^{1-2/d}\mid |\C_n(r,\rho)|> L\big)=1, 
\end{equation}
where the infimum is taken over all triples $(r,\rho,L)$ satisfying \eqref{cond-folding}. 
\end{theorem}
Let us stress that condition~\eqref{cond-folding} is optimal in the following sense. 
Concerning the first part, 
just recall that in 
time $n$, the walk typically fills balls with an occupation density of order $r^{2-d} \log n$, and  
for the second part, which is only needed for the lower bound in \eqref{ineq-folding}, 
note that one needs at least $n\ge \rho r^dL$, for the set $\C_n(r,\rho)$ to be non-empty.  
Let us also mention here that 
we obtain a similar result as Theorem \ref{prop-folding},
where instead of recording the time spent
in small cubes, we count the number of visited sites, 
see Proposition \ref{prop-LB} for details.

\begin{remark}\label{rem-level}
\emph{Note that the result is interesting on its own right even for $r=1$, 
in which case it concerns the so-called {\it level-sets} 
of the local times, that is the sets of the form 
\begin{equation*}
\cL_n(\rho):= \{z\in \Z^d\, :\, \ell_n(z)>\rho\}.
\end{equation*}
Specializing Theorem \ref{prop-folding} to these sets gives
that for $\rho\ge C\cdot \log(n)$ and $ n\ge C\rho\cdot L$, 
\begin{equation*}
\exp(-\underline \kappa \cdot\rho \cdot L^{1-2/d})\le 
\bP(|\cL_n(\rho) |> L )\le \exp(-\overline \kappa \cdot 
\rho \cdot L^{1-2/d}).
\end{equation*}
Furthermore, asymptotically as $n$ goes to infinity, 
conditionally on being non-empty, the shape of $\cL_n(\rho)$ is ball-like
in the following sense. There is $A>0$, such that for $\rho_n,L_n$
satisfying $\rho_n\ge C\cdot \log(n)$ and $ n\ge C\rho_n\cdot L_n$
\begin{equation*}
\lim_{n\to\infty} \bP\big(\capa(\cL_n(\rho_n))\le 
A\cdot |\cL_n(\rho_n)|^{1-2/d}\mid |\cL_n(\rho_n)|> L_n\big)=1.
\end{equation*}
}
\end{remark}

\begin{remark}\label{rem-simple}
\emph{
For simplicity, we focus here on the case of simple random walk, 
but our results would likely adapt to more general setting. 
}
\end{remark}

\paragraph{Historical Account.} Let us put our results into perspective.
Capacity appears as a central object in many remarkable studies, and  
we would like to highlight some of them. In the thirties, Wiener introduces 
his celebrated test, where the electrostatic capacity plays the key role, 
and is adapted to random walk context by It\^o and McKean much 
later~\cite{IK}. In the forties, Kakutani \cite{K44} discovers that
a compact set of $\R^d$, is hit by Brownian motion with
positive probability, if and only if it has positive electrostatic
capacity. 
Much later, Kesten \cite{Kes90} bounds the growth rate
of diffusion limited aggregation (DLA), a celebrated model 
of discrete random growth on $\Z^d$ where
sites in the boundary of the cluster are chosen 
according to the harmonic measure (of the boundary of the cluster). 
For doing so Kesten
introduces a martingale whose compensator is the sum of inverses 
of capacities of the growing cluster. This in itself is 
inspiring: understanding the growth of the
capacity of the cluster plays a key role in understanding 
the reinforcement phenomenon behind the ramified tree-like
shape of DLA (see also \cite{LT19} for a related model). Finally, ten
years ago, Sznitman \cite{S10} introduced a model called random
interlacements which is a homogeneous Poisson point process on $\Z^d$
such that the number of trajectories 
(the points of the process) hitting a given
compact set $K$ is a Poisson random variable with mean  $u\cdot \capa(K)$,
and whose hitting sites distribution on $K$ is according to the 
harmonic measure of $K$.
The model of random interlacements proves (or is conjectured) to be adapted to the study
of many phenomena where a random walk realizes atypically high densities:
(i) either by reducing its range, and in a certain 
regime this is the {\it Swiss Cheese} problem (see \cite{BBH}),
(ii) or by disconnecting the ball $B_n(0)$ from
from the complement of $B_{2n}(0)$, and many more sophisticated events, 
see in particular \cite{S17, NS20, S20}.

Concerning the deviations for local times, a rich literature exists on 
Large Deviation for the field of renormalized local times, 
initiated by Donsker and Varadhan~\cite{DV}, 
or for self-intersection local times, see \cite{Chen} 
and references therein.  
However, it seems that not much is known concerning 
the deviations of local times of a random walk on a fixed finite set 
(except of course when this set is made of only one point).

The paper is organized as follows. In Section~\ref{Sec.cap} we recall some basic facts on the capacity. 
Section~\ref{sec-tech} contains our main technical
novelty: the proof of Theorem \ref{prop-main}. 
In Section \ref{sec-AS15}, we prove Theorem \ref{theo-AS15}, and
introduce a related result Theorem ~\ref{theo.1} 
of a similar flavor, but dealing with the local times of sites.
Finally, in Section \ref{sec-folding} we prove Theorem \ref{prop-folding}. 
The proof is divided into a short upper bound, and a technical lower
bound in Section~\ref{subsec.lower} where we actually state
Proposition~\ref{prop-LB} which deals with
the (slightly more difficult) problem of covering a certain partition of space, rather than with local times.

\section{Preliminaries on capacity}\label{Sec.cap}
We recall here some alternative definitions of the capacity, and refer to \cite{LL} and \cite{S12} for proofs of these standard facts.  
The first alternative and equivalent definition is in terms of hitting time, rather than escape probabilities: 
\begin{equation}\label{def-DLA}
\capa(\Lambda)=\lim_{\|z\|\to\infty} \frac{1}{G(z)} \bP_z\big(
H^+_{\Lambda}<\infty\big), 
\end{equation}
where $G$ is Green's function: 
$$G(z) := \sum_{n\ge 0} \bP(S_n=z).$$
A third equivalent way to define the capacity is given by the 
variational formula
\begin{equation}\label{def-1}
\frac{1}{\capa(\Lambda)}=\inf\big\{ \sum_{x\in \Lambda}\!
\sum_{y\in \Lambda} G(x-y) \mu(x)\mu(y):\ 
\mu \text{ probability  on }\Lambda\big\}.
\end{equation}
The infimum is reached for the
{\it equilibrium measure} $e_\Lambda$, 
defined for $x\in \Lambda$ by $e_\Lambda(x)=
\bP_{x}(H^+_\Lambda=\infty)/\capa(\Lambda)$. 

As we already recalled, the capacity of a ball $B_r(x)$ 
is of order $r^{d-2}$, and more generally, 
there exists a constant $a>0$, such that for any $\Lambda\subset \Z^d$,
\begin{equation}\label{bound.cap.lower}
\capa(\Lambda) \ge a |\Lambda |^{1-2/d},
\end{equation}
When applied to a union of disjoint balls, this gives
\begin{equation}\label{bounds-capa}
\capa(B_r(\C) )\ge a \cdot r^{d-2} |\mathcal C|^{1-2/d},
\end{equation}
for some (possibly different) constant $a>0$. 
This bound cannot be improved. 
Looking now for an upper bound of the capacity of a union of balls, 
subadditivity of the capacity gives that it is always bounded
(up to constant) by the number of balls times $r^{d-2}$. However, 
one can improve this crude bound using \reff{def-DLA} yielding
\begin{equation}\label{upper.cap.balls}
\capa(B_r(\C) )\le A\cdot r^{d-2}\cdot \capa(\C),
\end{equation}
for some constant $A>0$ (independent of $r$ and $\C$).

We shall also need the following lemma. 
For $r\ge 1$, we denote by $\mathcal X_r$ the set of finite $\C\subset \Z^d$, whose points are all at distance at least $4r$ one from each other. 
\begin{lemma}\label{lem.Xr}
There exists a constant $c>0$, such that for any finite $\C\subset \Z^d$, there exists a subset $\C'\subseteq \C$, with $\C'\in \mathcal X_r$, satisfying
$$\capa(B_r(\C')) \ge c\cdot \capa(B_r(\C)).$$
\end{lemma}
\begin{proof}
We define recursively $\mathcal C_n\subseteq \C$, for $1\le n\le |\C|$ as follows. First pick a point $x_1$ in $\C$, and set $\C_1 := \{x_1\}$. 
Then assuming $\C_n$ has been defined for some $n<|\C|$, define $\C_{n+1}$ as the union of $\C_n$ and a point of $\C\setminus (\cup_{x\in \C_n}B_{4r}(x))$, if this set is nonempty. Otherwise, set $\C_{n+1}:= \C_n$. Define $\C'$ as the set one eventually obtains.
Note that by construction $\C'\in \mathcal X_r$. 

We express now the hitting time of $B_r(\C')$ and use that
$B_r(\C)\subset B_{4r}(\C')$ to obtain for any $z\in \Z^d$,
\begin{align*}
\bP_z(H(B_r(\C'))<\infty)=& \bP_z(H(B_{4r}(\C'))<\infty)\times
\bP_z(H(B_r(\C'))<\infty\ \big|\ H(B_{4r}(\C'))<\infty)\\
\ge&
\bP_z(H(B_{r}(\C))<\infty)\times
\bP_z(H(B_r(\C'))<\infty\ \big|\ H(B_{4r}(\C'))<\infty).
\end{align*}
Since after arriving on the boundary of a ball $B_{4r}(x)$, for $x\in \C'$,
the random
walk hits $B_r(x)$ with a positive probability, say $c$ independent of
$r$ and $\C'$, we obtain
\[
\bP_z(H(B_r(\C'))<\infty\ \big|\ H(B_{4r}(\C'))<\infty)> c.
\]
The proof ends as we recall \reff{def-DLA}, normalize
$\bP_z(H(B_r(\C'))<\infty$ by $G(z)$ and take $z$ to infinity.
\end{proof}

\section{Proof of Theorem~\ref{prop-main}}\label{sec-tech}
The proof is an instance of the probabilistic method: 
we define an appropriately chosen random subset of $\C$, 
and show that it satisfies the desired constraints 
with nonzero probability. 

We start with the proof in the case $r=1$, 
which we think is instructive and more transparent. 

\underline{Case $r=1$}. We need to show that in any finite set $\Lambda\subseteq \Z^d$, there exists a subset $U$, whose capacity and cardinality are both of the order of the capacity of $\Lambda$. The proof is an instance of the probabilistic method.
Indeed, we build a random set $\cU$ which satisfies the desired constraints 
with positive probability.

First, choose a family of i.i.d. trajectories
$(\gamma^x,x\in \Z^d)$ with the same law as the walk $S=\{S_n\}_{n\ge 0}$ starting from the origin, and denote their 
joint law by $\bP$. The hitting time of $\Lambda$ by a (random) path $\gamma: \N\to \Z^d$  is denoted by 
$H_\Lambda(\gamma)$, the return time to $\Lambda$ by 
$H_\Lambda^+(\gamma)$, and set $\gamma_x=\gamma^x+x$.
Now, the random set $\cU$ is
\begin{equation*}
\cU:=\{x\in \Lambda:\ H_\Lambda^+(\gamma_x)=\infty\}.
\end{equation*}
Note that the volume of $\cU$ is a sum of independent Bernoulli random variables,
and thus
\begin{equation*}
\bE[|\cU|]=\sum_{x\in \Lambda} \bP\big(H_\Lambda^+(\gamma_x)=\infty\big)=
\capa(\Lambda),\quad\text{and}\quad
\var(|\cU|)\le \capa(\Lambda).
\end{equation*}
Thus $|\cU|$ is concentrated around its mean and by Chebychev's inequality
\begin{equation}\label{ineq-5}
\bP\big( |\cU|< \frac{1}{2} \bE[|\cU|]\big)\le \frac{4}{\capa(\Lambda)}
,\quad\text{and}\quad
\bP\big( |\cU|> 2\bE[|\cU|]\big)\le \frac{1}{\capa(\Lambda)}.
\end{equation}
We can assume $\capa(\Lambda)>16$, (as for sets with bounded capacity one can always choose $\alpha$ small enough) so that
\reff{ineq-5} reads
\begin{equation}\label{ineq-vol}
\bP\big( 2\capa(\Lambda)\ge 
|\cU|\ge \frac{1}{2} \capa(\Lambda)\big)\ge \frac{2}{3}.
\end{equation}
Now, we show that $\capa(\cU)$ is of order its volume. By \reff{def-1},
if we choose for $\mu$ the uniform measure on $\cU$, then
\begin{equation}\label{ineq-uni}
\frac{\capa(\cU)}{|\cU|}\ge \Big( \frac{1}{|\cU|} \sum_{x,y\in \cU}
G(x-y)\Big)^{-1}.
\end{equation}
Let us compute the expression on the right hand side of \reff{ineq-uni}.
\begin{equation*}
\begin{split}
 \sum_{x,y\in \cU}G(x-y) & = \sum_{x\in \Lambda} \sum_{y\in \Lambda}
\1\{H_\Lambda^+(\gamma_x)=\infty,\, H_\Lambda^+(\gamma_y)=\infty\}\cdot  G(x-y)\\
& = G(0)\cdot |\cU|+\sum_{x\in \Lambda}\sum_{y \in \Lambda\bs\{x\}} 
\1\{H_\Lambda^+(\gamma_x)=\infty,\, H_\Lambda^+(\gamma_y)=\infty\}\cdot G(x-y).
\end{split}
\end{equation*}
Note that if $x\not=y$, then $\gamma_x$ and $\gamma_y$ are independent.
Therefore,
\begin{equation*}
\bE\left[ \sum_{x,y\in \cU}G(x-y)\right]\le G(0)\cdot \bE[|\cU|]+
\sum_{x,y\in\Lambda}\bP\big(H_\Lambda^+(\gamma_x)=\infty\big)
G(x-y) \bP\big(H_\Lambda^+(\gamma_y)=\infty\big).
\end{equation*}
By a last passage decomposition (see Proposition 4.6.4 in \cite{LL}), for $x\in \Lambda$,
\begin{equation*}
1= \bP\big(H_\Lambda(\gamma_x)<\infty\big) =
\sum_{y\in\Lambda}G(x-y) \bP\big(H_\Lambda^+(\gamma_y)=\infty\big).
\end{equation*}
Thus,
\begin{equation*}
\bE\left[\sum_{x,y\in \cU}G(x-y)\right]\le \big(G(0)+1\big)\cdot \capa(\Lambda),
\quad\text{and}\quad
\bP\left(\sum_{x,y\in \cU}G(x-y)\le 
4(G(0)+1)\cdot \capa(\Lambda)\right) \ge \frac{3}{4}.
\end{equation*}
Together with \reff{ineq-vol}, we obtain
\begin{equation}\label{ineq-9}
 \bP\left(2\capa(\Lambda)\ge  |\cU|\ge \frac{1}{2}\capa(\Lambda),\ 
\sum_{x,y\in\cU}G(x-y)\le 4\big(G(0)+1\big)\cdot \capa(\Lambda)\right)
\ge \frac{5}{12}.
\end{equation}
By \reff{ineq-uni} and \reff{ineq-9}, we deduce that for some $\alpha>0$, 
\begin{equation}\label{ineq-main}
\bP\left(2\capa(\Lambda)\ge |\cU|\ge  \frac{1}{2}\capa(\Lambda),\ 
\capa(\cU)\ge\alpha\cdot\capa(\Lambda)\right) 
\ge \frac{5}{12}.
\end{equation}
Thus, we conclude that \reff{ineq-cor} holds for a random set $\cU$, when $r=1$.

We now prove the general case  by refining the previous argument.

\underline{General case $r\ge 1$}. 
The proof follows the same steps after we 
choose an appropriate random subset of the set of centers $\C$. First, by Lemma \ref{lem.Xr} one can assume that $\C\in \mathcal X_r$ (i.e. 
that points of $\C$ are all at distance at least $4r$ one from each other). 
For simplicity, for $r>0$, we set $\Lambda_r=B_r(\C)$, and 
$V_r$ to be the complement of $B_{2r}(\C)$.
We need now the hitting time of $\Lambda_r$
after exiting $B_{2r}(\C)$. For a trajectory $\gamma$, define
\begin{equation*}
H_{\Lambda_r}^r(\gamma)=\inf\{k> H_{V_r}(\gamma):\gamma(k)\in \Lambda_r\}.
\end{equation*}
Then choose a family of i.i.d. trajectories
$(\gamma^x,x\in \Z^d)$ with the same law as $S$, denote the
joint law by $\bP$, and set $\gamma_x=\gamma^x+x$.
Our random set reads now 
\begin{equation*}
\cU:=\{x\in \C:\ H_{\Lambda_r}^r(\gamma_x)=\infty\}.
\end{equation*}
Thus, each center $x\in \C$ is kept in $\cU$ if a random walk launched
from $x$ escapes $\Lambda_r$ after exiting $B_{2r}(\C)$. The reason
to force first to exit $B_{2r}(\C)$ stems from the following
lemma taken from \cite{AS15}. 

\begin{lemma}\label{lem-1}
There is $\theta>1$, such that for any $r\ge 1$, $\C\in \XX_r$, and  $x\in \C$,
\begin{equation}\label{cor-3}
\theta \bP\big(H^r_{\Lambda_r}(\gamma_x)=\infty\big)\ge
\frac{1}{r^{d-2}}\sum_{y\in \partial B_r(x)} 
\bP\big(H^+_{\Lambda_r} (y+S)=\infty\big)\ge 
\frac{1}{\theta} \bP\big(H^r_{\Lambda_r}(\gamma_x)=\infty\big).
\end{equation}
\end{lemma}

Now, note that $|B_r(\cU)|/|B_r|$ is a sum of $|\C|$ independent
Bernoulli random variables, and therefore 
\begin{equation*}
\var(|B_r(\cU)|)\le |B_r|\cdot \bE[|B_r(\cU)|].
\end{equation*}
Furthermore, thanks to Lemma~\ref{lem-1}, there are
positive constants $c_1$ and $c_2$, such that
\begin{equation*}
c_1 r^2\cdot \capa(B_r(\C))\le \bE[|B_r(\cU)|]=|B_r|\cdot\sum_{x\in \C}
\bP\big(H^r_{\Lambda_r}(\gamma_x)=\infty\big)\le
c_2 r^2\cdot \capa(B_r(\C)).
\end{equation*}
It follows that for a positive constant $c_3$, 
\begin{equation*}
\bP\left( \frac{1}{2} \bE[|B_r(\cU)|]\le |B_r(\cU)|\le 2  \bE[|B_r(\cU)|] \right)\ge 1-  c_3
\frac{r^{d-2}}{\capa(B_r(\C))}. 
\end{equation*}
We can assume that $\capa(B_r(\C))\ge 4c_3 r^{d-2}$ (as otherwise we conclude by taking $\alpha<1/(4c_3)$), in which case it follows that 
\begin{equation*}
\bP\left( \frac{1}{2} \bE[|B_r(\cU)|]\le |B_r(\cU)|\le 2  \bE[|B_r(\cU)|] \right)\ge \frac{3}{4}.
\end{equation*}
Thus, with probability larger than 3/4, the random set $\cU$ 
satisfies $(ii)$ of \reff{ineq-cor}. Let us check now $(i)$.
By \reff{def-1}, we obtain a lower bound on $\capa(B_r(\cU))$
as we choose a measure on $B_r(\cU)$. Taking the uniform measure on the {\it boundary of $B_r(\cU)$} gives
\begin{equation}\label{cor-8}
\frac{1}{(|\partial B_r|\cdot |\cU|)^2}\sum_{x,x'\in \cU}\sum_{y\in \partial
B_r(x)}\sum_{y'\in \partial B_r(x')} G(y-y')\ge \frac{1}{\capa(B_r(\cU))}.
\end{equation}
We need to show that the left hand side of \reff{cor-8} is
smaller than $1/(r^{d-2}|\cU|)$. First, we treat the case
$x'=x$. Note that
by Green's function asymptotic \reff{asymp-G}, 
there is $c>0$, such that
\[
\forall y\in \partial B_r(x)\qquad
\sum_{y'\in \partial B_r(x)} G(y-y')\le c\cdot  r.
\]
Thus, as we further
sum over $y\in \partial B_r(x)$, and $x\in \cU$, we obtain
\begin{equation}\label{cor-9}
\sum_{x\in \cU} \sum_{y\in \partial B_r(x)}
\sum_{y'\in \partial B_r(x)} G(y-y')\le c\cdot r\cdot r^{d-1}
\cdot |\cU|.
\end{equation}
Now, to deal with the terms with $x'\not= x$, we take
expectation first, and we bound $G(y-y')$
by $c_4\cdot G(x-x')$ uniformly in $y\in \partial B_r(x)$ and 
$y'\in \partial B_r(x')$. Therefore,
\begin{equation*}
\begin{split}
\bE\left[\sum_{x\not=x'\in \cU} \sum_{y\in \partial
B_r(x)}\sum_{y'\in \partial B_r(x')} G(y-y')\right]& \le  c_4|\partial B_r|^2
\cdot \bE\big[\sum_{x\not=x'\in \cU} G(x-x')\big]\\
& \le  c_4|\partial B_r|^2 \sum_{x\not= x'\in \C}\!\bP\big(H^r_{\Lambda_r}
(\gamma_x)=\infty\big)  G(x-x') 
\bP\big(H^r_{\Lambda_r} (\gamma_{x'})=\infty\big)\\
& \le  c_5 (r^{d-1})^2\bE[|\cU|] \sup_{x\in \C}
 \sum_{x'\not= x} G(x-x')\bP\big(H^r_{\Lambda_r} (\gamma_{x'})=\infty\big).
\end{split}
\end{equation*}
By using \reff{cor-3} of Lemma~\ref{lem-1}, 
and a last passage decomposition we have for a constant $c_6>0$, and any $x\in \C$, 
\begin{equation*}
\begin{split}
1&=\bP\big(H_{\Lambda_r}(\gamma_x)<\infty\big)
\ge \sum_{\substack{x'\in \C \\ x'\neq x}}
\sum_{y\in \partial B_r(x')} G(x-y) 
\bP(H^+_{\Lambda_r}(\gamma_y)=\infty\big)\\
&\ge c_6 r^{d-2} \sum_{\substack{x'\in \C \\ x'\neq x}} G(x-x') 
\bP\big(H^r_{\Lambda_r} (\gamma_{x'})=\infty\big).
\end{split}
\end{equation*}
This implies that for a constant $c_7>0$, 
\begin{equation*}
\bE\left[\sum_{x\not=x'\in \cU} \sum_{y\in \partial
B_r(x)}\sum_{y'\in \partial B_r(x')} G(y-y')\right]\le 
c_7 r^d \cdot \bE[|\cU|].
\end{equation*}
Chebychev's inequality now allows us to conclude as in
the proof of the case $r=1$.

\section{Proof of Theorem~\ref{theo-AS15}}\label{sec-AS15}
We start with a 
variant of Theorem~\ref{theo-AS15} which deals with the event of 
visiting a certain number of times each site of a set, and connect the
probability of such an event with the capacity of the set. 
Let $q:=\bP(H_0^+<\infty)$,
and for a finite set $\Lambda\subset \Z^d$, and $z\in \Lambda$, let 
$q_z = q_{z,\Lambda}: = \bP_z(H_\Lambda^+<\infty)$. 
\begin{theorem}\label{theo.1}
Assume $d\ge 3$, and let $\Lambda$ be a finite subset of $\Z^d$. 
Then, for any set of nonnegative integers $(n_z)_{z\in \Lambda}$, 
\begin{equation}\label{theo.1.1}
\mathbb P(\ell_\infty(z) \ge n_z \quad \forall z\in \Lambda) \le \frac{\prod_{z\in \Lambda} q_z^{n_z}}{\min_{z\in \Lambda}q_z}.
\end{equation}
In particular for all $t\ge 1$, 
\begin{equation}\label{theo.1.2}
\mathbb P(\ell_\infty(z) \ge t \quad \forall z\in \Lambda) 
\le \frac{1}{q} \exp\big(-t\cdot \capa(\Lambda)\big). 
\end{equation}
\end{theorem}
Both Theorems~\ref{theo-AS15} and \ref{theo.1} use improvements
of the proof of Proposition 1.7 of \cite{AS15}. We present
a simple self-contained proof of Theorem~\ref{theo.1}, 
and note that it works in fact for any random walk. 

\subsection{Proof of Theorem~\ref{theo.1}} 
The proof proceeds by induction on $N:= \sum_{z\in \Lambda} n_z$. 
Note that it is important in the proof to allow some integers $n_z$ to be equal to $0$. 
If $N=0$ or $N=1$, there is nothing to prove, since the right-hand side of \eqref{theo.1.1} is larger than or equal to $1$ in this case. Assume now that the result is true for 
any sequence $(n_z)_{z\in \Lambda}$, with $\sum_{z\in \Lambda} n_z \le N$, for some $N\ge 1$, and consider another sequence (which we still denote by $(n_z)_{z\in \Lambda}$) satisfying $\sum n_z = N+1$. 
If $0\in \Lambda$, and $n_0\ge 2$, we write (recalling that in the definition of local times, the time $0$ is taken into account), with $\Lambda^*:=\Lambda\setminus \{0\}$, 
\begin{equation}\label{AS15-best}
\begin{split}
\bP(\ell_\infty(z) \ge n_z\ \forall z\in \Lambda) & = \sum_{y\in \Lambda:\, n_y\ge 1}  \bP(H_\Lambda^+<\infty, S(H_\Lambda^+) =y)\cdot 
\bP_y(\ell_\infty(z) \ge n_z\ \forall z\in \Lambda^*,\,  \ell_\infty(0) \ge n_0-1)\\
& \le \sum_{y\in \Lambda:\, n_y\ge 1} \bP(H_\Lambda^+<\infty, S(H_\Lambda^+) =y)\cdot \frac{(\prod_{z\in \Lambda^*} q_z^{n_z} )\cdot q_0^{n_0-1}}{\min_{z\in \Lambda} q_z} \\
& \le \bP(H_\Lambda^+<\infty) \frac{(\prod_{z\in \Lambda^*} q_z^{n_z} ) q_0^{n_0-1}}{\min_{z\in \Lambda} q_z} = \frac{\prod_{z\in \Lambda} q_z^{n_z} }{\min_{z\in \Lambda} q_z}, 
\end{split}
\end{equation}
using the induction hypothesis at the second line. This proves the induction step in the case when $0\in \Lambda$ and $n_0\ge 2$. If $0\in \Lambda$ and $n_0=1$, 
we write similarly 
\begin{align*}
\bP(\ell_\infty(z) \ge n_z\ \forall z\in \Lambda) & = \sum_{y\in \Lambda^*:\, n_y\ge 1}  \bP(H_\Lambda^+<\infty, S(H_\Lambda^+) =y)\cdot 
\bP_y(\ell_\infty(z) \ge n_z\ \forall z\in \Lambda^*)\\
& \le \sum_{y\in \Lambda^*:\, n_y\ge 1} \bP(H_\Lambda^+<\infty, S(H_\Lambda^+) =y)\cdot \frac{\prod_{z\in \Lambda^*} q_z^{n_z} }{\min_{z\in \Lambda} q_z} \\
& \le \bP(H_\Lambda^+<\infty) \frac{\prod_{z\in \Lambda^*} q_z^{n_z} }{\min_{z\in \Lambda} q_z} = \frac{\prod_{z\in \Lambda} q_z^{n_z} }{\min_{z\in \Lambda} q_z}, 
\end{align*}
proving as well the induction hypothesis. Finally, when $0 \notin \Lambda$, or when $0\in \Lambda$ and $n_0=0$, one can simply bound the probability on the left hand side above by the probability to 
hit a point $y\in \Lambda$ with $n_y\ge 1$, and then by  the Markov property we are back to the previous situation. Altogether this concludes the proof of the first assertion \eqref{theo.1.1}.

The second assertion \eqref{theo.1.2} follows immediately from \eqref{theo.1.1}, using that for any $z\in \Lambda$, 
$$q_z = 1- \bP_z(H_\Lambda^+ = \infty) \le \exp(-\bP_z(H_\Lambda^+= \infty)).$$ 

\subsection{Proof of Theorem~\ref{theo-AS15}}
First, using Lemma~\ref{lem.Xr}, one can assume that all points of $\C$ are at distance at least $4r$ one from each other, as stated in \cite{AS15}.
The proof of Proposition 1.7 of \cite{AS15} then shows 
that for some positive constants $\kappa$ and $c$, 
for all $r\ge 1$, $\rho>0$, and $\C\in \mathcal X_r$, 
\begin{equation}\label{old-as15}
\bP\big(\forall x\in \C,\ \ell_\infty(B_r(x))> \rho r^d\big)\le
\exp\big(c|\C|-\kappa\cdot \rho\cdot \capa(B_r(\C)\big).
\end{equation}
Indeed, instead of considering case of equalities in (3.3) of
\cite{AS15}, it is enough to consider an induction step with a number
of excursions larger than or equal to $\{n_z,z\in \C\}$, as in
\reff{AS15-best} of the previous proof (that is replace $=$ with $\ge$
in (3.3) of \cite{AS15}).
Then, Theorem \ref{prop-main} gives the existence of a 
subset $U\subseteq \C$, with $\capa(B_r(U))$ of the same order 
as both $r^{d-2}\cdot |U|$ and $\capa(B_r(\C))$.
This allows to remove the combinatorial factor in \eqref{old-as15}, 
using the hypothesis \eqref{best-as15-hyp}, and this concludes the proof of the theorem.

\section{Application to Folding, and proof of Theorem \ref{prop-folding}}\label{sec-folding}
The proof of Theorem \ref{prop-folding} is divided in two parts. In the first part (see Subsection \ref{subsec.upper} below), we show that for some positive constants $\tilde \kappa$  and $A_0$,  for any $A>A_0$, any $n\ge 1$, and any $(r,\rho,L)$ satisfying \eqref{cond-folding}, 
\begin{equation}\label{bounds-folding}
\bP\left(|\C_n(r,\rho)|>L,\, \capa(\V_n(r,\rho))> A|\V_n(r,\rho)|^{1-2/d}\right)\le \exp(-\tilde \kappa A\cdot  \rho\cdot  r^{d-2} L^{1-2/d}).
\end{equation} 
In the second part (see Subsection \ref{subsec.lower} below) we prove the lower bound in \eqref{ineq-folding}. Note that altogether this gives \eqref{shape-folding} as well, and thus proves Theorem \ref{prop-folding}.

\subsection{The upper bound: proof of \reff{bounds-folding}}\label{subsec.upper}
We introduce the notation $Q_r(U)$ for $\cup_{x\in U} Q_r(x)$,
and we use Theorem~\ref{prop-main}, with the 
condition~\reff{cond-folding} and then \reff{best-as15} as follows.
\begin{equation*}
\begin{split}
\bP\big(|\C_n(r,\rho)|>& L,\ \capa(\V_n(r,\rho))\ge A\cdot |\V_n(r,\rho)|^{1-\frac 2d}\big)\le 
\sum_{k>L}  \bP\big(|\C_n(r,\rho)|=k, \capa(\V_n(r,\rho))
\ge A\cdot r^{d-2}k^{1-2/d}\big)\\
&\le \sum_{L<k\le n} \bP\big(\exists\ \cU\subset [-n,n]^d:\ 
k\ge |\cU|> \alpha A k^{1-2/d},\ \capa(Q_r(\cU))\ge 
\alpha r^{d-2} |\cU|\big)\\
&\le \sum_{L<k\le n} \sum_{\alpha A k^{1-2/d}<i\le k} 
\bP\big(\exists\ \cU\subset [-n,n]^d:\ |\cU|=i,\  
\capa(Q_r(\cU))\ge \alpha r^{d-2} i\big)\\
&\le  \sum_{L<k\le n} \sum_{\alpha A k^{1-2/d}<i\le k}
c (2n)^{d\cdot i}\cdot i!\cdot
\exp\big(-\kappa \alpha \rho r^{d-2}\cdot i \big)\\
&\le\ c\sum_{k>L} 
\exp\big(-\tilde \kappa A\cdot \rho r^{d-2} k^{1-2/d}\big)
\le\ c  \exp\big(-2\tilde \kappa A\cdot \rho r^{d-2} L^{1-2/d}\big). 
\end{split}
\end{equation*}
The combinatorial factor $(2n)^{d\cdot i}\cdot i!$ was swallowed
after using the condition that $r^{d-2}\cdot \rho> C \log(n)$,
and choosing $A$ large enough.

\subsection{Lower bound}\label{subsec.lower}
In this subsection, we establish a result which slightly differs from the lower bound in \eqref{ineq-folding}, and 
deals with covering rather than occupation. For this purpose we introduce for $n\ge 1$, the range
of the walk $\cR_n:=\{S_0,\dots,S_n\}$, and for any  $r\ge 1$ and $\rho\in [0,1]$,  
\[
\widetilde \C_n(r,\rho):= \{z\in r\Z^d:\ |\cR_n\cap Q_r(z)|\ge \rho |Q_r|\}.
\]
Our result is as follows. 
\begin{proposition}\label{prop-LB}
There exist positive constants $c$ and $C$, 
such that for any $n\ge 1$, $r>0$, $\rho\in (0,1/2)$, and $L\ge 1$, satisfying 
$$\rho r^{d-2}\ge 1,\quad \text{and} \quad  n\ge C\rho r^d L,$$
one has 
\begin{equation*}
\bP\big(|\widetilde C_n(r,\rho)|\ge  L\big) 
\ge c\, \exp(-c\, \rho r^{d-2}  L^{1-\frac 2d}).
\end{equation*}
\end{proposition}
Thus this result is exactly the same as the lower bound in \eqref{ineq-folding}, but for this new set $\widetilde \C_n(r,\rho)$. 
Since the proof for $\C_n(r,\rho)$ is easier, we will only provide the proof for the set $\widetilde \C_n(r,\rho)$.  
Note also that since $\widetilde \C_n(r,\rho) \subseteq \C_n(r,\rho)$, the upper bound in \eqref{ineq-folding}, as well as \eqref{shape-folding}, also hold for the set $\widetilde \C_n(r,\rho)$. 
   
\begin{remark}\emph{The hypothesis $\rho<1/2$ in Proposition \ref{prop-LB} could be replaced 
by $\rho<1-\eta$, for any fixed constant $\eta>0$, 
and the constants $c$ and $C$ would then
depend on $\eta$. However, when $\rho$ gets close to $1$, we fall in another regime, and for instance when $\rho=1$  
an extra $\log r$ factor is needed in the exponential
(and in the time needed to achieve the covering).}
\end{remark}
\begin{proof}
The scenario we choose to produce the desired event is to localize the
walk long enough in a ball so that its occupation density is $\rho$. 
It is convenient to transform
localization into a statement about excursions.
We call {\it excursion} from a set $U$ to a set $V$,
the part of the random walk path starting from a point in $U$ 
up to its hitting time of $V$ which is maximal 
(for the inclusion of paths).
To produce $L$ boxes of side $r$, and filled to a density above $\rho$,
we expect to localize the walk in a ball of radius $R$ for a time $T$ with
\[
R=\lfloor  L^{1/d}r\rfloor,\quad\text{and}\quad
T=\lfloor C_1\rho R^d\rfloor\quad(\text{and $C_1$ large enough}).
\]
Let now $\NN_R$ be the number of excursions from $\partial Q_{2R}$ 
to $\partial Q_{4R}$ before exiting $Q_{8R}$, and consider 
\begin{equation}\label{def-A}
A:=\{\NN_R\ge N \},\quad\text{with}\quad 
N=\lfloor C_2\rho R^{d-2}\rfloor .
\end{equation}
Now, for any $C_2$ (that is any number of excursions $N$), 
one makes the event $A$ typical by choosing
$C_1$ large (that is a large localization time $T$), and assume $n\ge T$.
On event $A$, we define $X=(X_1,\dots,X_N)$ and $Y=(Y_1,\dots,Y_N)$, 
as respectively the starting and ending points of the $N$ 
first excursions from $\partial Q_{2R}$ to $\partial Q_{4R}$.
Then given some fixed ${\bf x}=(x_1,\dots,x_N)\in \partial Q^N_{2R}$
we let $\bP_{{\bf x}}$ be the law of $N$ excursions starting 
from $\{x_i,\ i\le N\}$, up to $\partial Q_{4R}$. 
We still denote by $Y$ the set of ending points 
of the $N$ excursions under $\bP_{{\bf x}}$. 
We let $M$ be the cardinality of the set $r\Z^d\cap Q_{R-r}$, 
and number its elements in some arbitrary order, say $v_1,\dots,v_M$.
We define
\begin{equation*}
B:= \left\{
\begin{array}{cc}
\text{The $N$ (first) excursions visit at least half 
of the boxes}\\
\text{$\{Q_r(v_i)\}_{i\le M}$, a fraction at least $\rho$ of their sites}
\end{array}
\right\}.
\end{equation*}
Assume for a moment that,
\begin{equation}\label{est-B}
\forall {\bf x} \in (\partial Q_{2R})^N,\qquad
\bP_{{\bf x}}(B)\ge 1/2.
\end{equation}
With $\sigma=\inf\{n\ge 1\, :\, S_n\in Q_{2R}\cup Q^c_{8R}\}$, we have
\begin{equation}\label{LB-1}
\bP(A\cap B,\, X={\bf x},\, Y={\bf y})=
\prod_{i=1}^{n-1} \bP_{y_i}\big(S(\sigma)=x_{i+1}\big)\cdot
\bP_{{\bf x}}(B,Y={\bf y}).
\end{equation}
Using Harnack's inequality (see \cite[Theorem 6.3.9]{LL}), for some constant $c_H>0$, for any $x\in \partial Q_{2R}$, 
\begin{equation}\label{LB-2}
\inf_{y}  \bP_y\big(S(\sigma)=x\big)\ge c_H 
\bP_{y^*}\big(S(\sigma)=x\big),\quad\text{with}\quad
y^*:=(4R,0,\dots,0).
\end{equation}
Thus, using that there is a positive lower bound (uniform in $R$) for 
the probability that a walk starting from $y^*$ 
hits $\partial Q_{2R}$ before $\partial Q_{8R}$, we have $c_1>0$,  such that
\begin{align}\label{AcapB}
\nonumber \bP(A\cap B) &=\sum_{{\bf x}}\sum_{{\bf y}}
\bP(A\cap B, X={\bf x},Y={\bf y}) \ge \sum_{{\bf x}}c_H^N
\prod_{i=1}^{n}\bP_{y^*}\big(S(\sigma)=x_i\big) \sum_{{\bf y}}
\bP_{{\bf x}}(B,Y={\bf y}) \\
& \ge  c_H^N \inf_{{\bf x}}\bP_{{\bf x}}(B) \sum_{{\bf x}} 
\prod_{i=1}^{n}\bP_{y^*}\big(S(\sigma)=x_i\big)
\ge \frac{c_H^N}{2}\prod_{i=1}^{n}\bP_{y^*}\big(S(\sigma)
\in B_{2R}\big)\ge e^{-c_1 N}.
\end{align}
Finally, define
\begin{equation*}
C:=\{\text{The walk makes at least $N$ excursions 
from $\partial Q_{4R}$ to $\partial Q_{2R}$ before time $T$}\}.
\end{equation*}
Using that on the event $A\cap C^c$, the walk spends 
a time at least $T$ in $Q_{8R}$,
we deduce that for some constant $c>0$, 
\begin{equation}\label{LB-3}
\bP(A\cap C^c)\le \exp(-c \frac{T}{R^2}).
\end{equation}
Then the proposition readily follows from 
\eqref{AcapB} and \eqref{LB-3}, once we choose $C_1c>2C_2c_1$
and use
\[
\bP\big( A\cap B\cap C\big)\ge \bP\big( A\cap B\big)-
\bP\big(A\cap C^c\big).
\]
We now prove \reff{est-B}.
We fix some ${\bf x}\in\partial Q^N_{2R}$, and
in the remaining part of the proof, we work under $\bP_{{\bf x}}$ .
We denote by $\cR^N$ the range produced by the $N$ excursions.
We note that it suffices to show that for $C_2$ large enough, 
one has for any $i\le M$,
\begin{equation}\label{LB-5}
\bP_{{\bf x}}\big(|\cR^N\cap Q_r(v_i)|> \rho |Q_r|\big)\ge 3/4.
\end{equation}
Indeed, letting $Z$ be the number of boxes whose fraction of visited sites exceeds $\rho$, \eqref{LB-5} shows that $\bE[Z] \ge (3/4) M$, and using also that $Z$ is bounded by $M$, 
it implies that $\bP(Z\le M/2)\le 1/2$, as wanted.

Thus, we are lead to prove \eqref{LB-5} for
$i\le M$. For a chosen $i\le M$, we introduce new notation.
Let $\NN_r$ be the number of excursions which hit $\partial Q_{2r}(v_i)$,
and $\G$ be the $\sigma-$field generated by $\NN_r$ 
and  the hitting points of $\partial Q_{2r}(v_i)$ by these excursions. Finally we let  
$\XX\subseteq Q_r(v_i)$ be the set of vertices visited by these excursions.
Since any vertex in $Q_r(v_i)$ has a probability 
of order $r^{2-d}$ to be visited by a walk starting 
from $\partial Q_{2r}(v_i)$, uniformely in its starting point, 
we have for some constant $c_0>0$, almost surely 
\begin{equation}\label{espX}
\bE_{{\bf x}}[|\XX|\mid \G] \ge 
\left(1 - (1-\frac{c_0}{r^{d-2}})^{\NN_r}\right)
\cdot|Q_r|\ge\Big(1-\exp\big(-c_0\frac{\NN_r}{r^{d-2}}\big)\Big)\cdot|Q_r|,
\end{equation}
using that $1-u \le e^{-u}$, for all $u\ge 0$.
We choose a large constant $K$ whose value will be fixed later. 
Since any excursion has a probability of order 
at least $(r/R)^{d-2}$ to hit $\partial Q_{2r}(v_i)$,
it is possible to choose $C_1$ (and $C_2$) so that 
for this chosen $K$,  
\begin{equation}\label{LB-6}
\bP_{{\bf x}}\big(\NN_r \ge K \rho r^{d-2}\big) \ge \sqrt{7/8}.
\end{equation}
We distinguish now a high and a low density regimes.
\paragraph{High density.} If $1-\exp(-c_0\ K\rho) \ge  \sqrt{7/8}$, then \eqref{espX} and 
\reff{LB-6} imply that 
\[
\bE_{{\bf x}}[|\XX|]\ge (7/8)\cdot |Q_r|.
\]
Using that $\XX\subseteq Q_r$, and as a consequence that $|\XX|$ is bounded by $|Q_r|$,  
we obtain with \reff{LB-5}, 
\[
\bP_{{\bf x}}(|\XX|< \rho |Q_r|)\le 
\bP_{{\bf x}}(|\XX|< |Q_r|/2) \le  1/4\qquad (\text{using}\
\rho<1/2). 
\]
\paragraph{Low density.} If $\rho$ is such that
$1-\exp(-c_0\ K\rho) < \sqrt{7/8}$,
then by \eqref{espX} we may choose $K$ large enough, so that 
on the event $\{\NN_r \ge K \rho r^{d-2}\}$, 
\begin{equation}\label{LB-7}
\bE_{{\bf x}}[|\XX|\mid \G]\ge \sqrt{K}\rho \cdot |Q_r|.
\end{equation}
Our strategy now is to use a second (conditional) moment method,
and show that the conditional variance of $\XX$ is small.
We denote with $\XX_1$ the set of pairs of vertices 
$(y,z)\in \XX\times \XX$, for which there exists an excursion going through both $y$ and $z$, and let $\XX_2$ be the complement of $\XX_1$ in 
$\XX\times \XX$ and note that $|\XX|^2=|\XX_1| + |\XX_2|$. 
Since for any $y\in Q_r(v_i)$ the mean number of 
vertices in $Q_r(v_i)$ which are visited by a walk starting 
from $y$ is of order $r^2$, one has for some constant $c>0$,
$$
\bE_{{\bf x}}[|\XX_1|\mid \G ]\le cr^2 \cdot \bE_{{\bf x}}[|\XX|\mid \G].
$$
Recalling $\rho r^{d-2}\ge 1$ and \eqref{LB-7}, 
it follows that on the event $\{\NN_r \ge K \rho r^{d-2}\}$, for a possibly larger constant $c$, 
\begin{equation}\label{XX1}
\bE_{{\bf x}}[|\XX_1|\mid \G]\le  \frac{c}{\sqrt K} 
\cdot \bE_{{\bf x}}[|\XX|\mid \G]^2.
\end{equation}
We can take $K$ such that $\sqrt K> 64c$.
It remains to bound the conditional mean of $|\XX_2|$ knowing $\G$. 
One can find a constant $K'>0$, such that
\[
\bP_{{\bf x}}(D) \ge 7/8,\quad\text{for}\quad
D:=\left\{ K \rho r^{d-2} \le \NN_r \le K' \rho r^{d-2}\right\}.
\]
We denote by $\cE_1,\dots,\cE_{\NN_r}$, the $\NN_r$ excursions hitting $Q_r(v_i)$ under $\bP_{{\bf x}}$. 
Fix some $y,z \in Q_r(v_i)$, and let
$$\cII_y:=\{k\le \NN_r\, :\, y\in \cE_k\}.$$
By definition, for any $k\le \NN_r$, 
$$\bP_{{\bf x}}(z\in \cE_k\mid \G,\, k\notin \cII_y) \le \frac{\bP_{{\bf x}}(z\in \cE_k\mid \G)}{\bP_{{\bf x}}(y\notin \cE_k\mid \G)} \le  \frac{\bP_{{\bf x}}(z\in \cE_k\mid \G)}{1-cr^{2-d}}\le \bP_{{\bf x}}(z\in \cE_k\mid \G) + \cO(r^{2(2-d)}) , $$
for some constant $c>0$.
As a consequence, on the event $D$, 
and for $r$ large enough,
\begin{align}\label{LB-8}
\nonumber \bP_{{\bf x}}\left(z\in \bigcup_{k\notin \cII_y} \cE_k\ \Big|\ \G ,\, \cII_y\right) & = 1-\prod_{k\notin \cII_y} \Big(1-\bP_{{\bf x}}(z\in \cE_k \mid \G, \,  \cII_y)\Big) \\
 \nonumber & \le 1-\prod_{k\notin \cII_y} \Big(1-\bP_{{\bf x}}(z\in \cE_k\mid \G) - \cO(r^{2(2-d)})\Big)\\
\nonumber & \le 1- \prod_{k\le \NN_r} \Big(1-\bP_{{\bf x}}(z\in \cE_k\mid \G)\Big) + \cO\left( \frac{\NN_r}{r^{2(d-2)}}\right)\\
& = \bP_{{\bf x}}(z\in \XX\mid \G) + \cO\left( \frac{\NN_r}{r^{2(d-2)}}\right),
\end{align}

where at the penultimate line we use that on $D$, 
and when $r$ is large enough, the term $\cO(\NN_r/r^{2(d-2)})$ 
can be made smaller than $1$.
Then on the event $D$, we get from \eqref{LB-8},
$$
\bP_{{\bf x}}((y,z) \in \XX_2\mid \G) \le  \left(\bP_{{\bf x}}(z\in \XX\mid \G) + \cO(\rho r^{2-d})\right) \cdot \bP_{{\bf x}}(y\in \XX\mid \G).$$
Summing over $y,z\in Q_r$, we deduce from \eqref{LB-7}, 
that on the event $D$,
$$
\bE_{{\bf x}}[|\XX_2|\mid\G ] \le \bE_{{\bf x}}[|\XX|\mid \G]^2 
(1+\cO(r^{2-d})).
$$
Combining this with \eqref{XX1}, we get for $r$ large enough,
$$
\var_{{\bf x}}(|\XX|\mid \G) = \bE_{{\bf x}}[|\XX_2|\mid\G ]+\bE_{{\bf x}}[|\XX_1|\mid\G ]-
\bE_{{\bf x}}[|\XX|\mid \G]^2
\le \frac{1}{32}\cdot \bE_{{\bf x}}[|\XX|\mid \G]^2.$$
Together with \eqref{LB-7}, it follows that for $r$ large enough, 
on the event $D$,
$$
\bP_{{\bf x}}(|\XX|\le \rho |Q_r|\mid \G) \le 
\bP_{{\bf x}}\left(|\XX|\le \frac 12\bE_{{\bf x}}[|\XX|\mid \G]\mid \G\right) \le \frac{4\var_{{\bf x}}(|\XX|\mid \G)}{\bE_{{\bf x}}[|\XX|\mid \G]^2} \le \frac 18.$$
Finally, using that $\bP_{{\bf x}}(D)\ge 7/8$, we obtain the desired
bound for $r$ large enough,
$$
\bP_{{\bf x}}(|\XX|\le \rho |Q_r|) \le 
\bP_{{\bf x}}(D^c) + \bP_{{\bf x}}(|\XX|\le \rho |Q_r|, \, D)\le 1/4,
$$
On the other hand, for small values of $r$, the result is immediate.
This concludes the proof of \eqref{est-B} and the Proposition.
\end{proof}

\noindent{\bf Aknowledgements} A.Asselah and B.Schapira
were supported by public grants overseen by the 
french National Research Agency, ANR SWiWS (ANR-17-CE40-0032-02) and ANR MALIN (ANR-16-CE93-0003).

\end{document}